\documentclass{amsart}
\usepackage{amsmath}
\usepackage{amsfonts}
\usepackage{amssymb}
\usepackage{graphicx}%
\newtheorem{mydef}{Definition}[section]
\newtheorem{myrem}{Remark}[section]
\newtheorem{mytheo}{Theorem}[section]
\newtheorem{mylem}{Lemma}[section]

\newtheorem*{mycorollary1.2}{Corollary 1.2*}

\usepackage{fancyhdr}

\usepackage{color}

\theoremstyle{remark}

\numberwithin{equation}{section}

 \allowdisplaybreaks[4]
\begin{document}

\title{Quantitative Estimates of The Singular Values of Random I.I.D. Matrices}

\pagestyle{fancy}

\fancyhf {} 


\fancyhead[CO]{\footnotesize  Deviation Inequalities for the singular values}

\fancyhead [CE]{\footnotesize G. DAI, Z. SU, AND H. WANG }

\fancyhead [LE]{\thepage}
\fancyhead [RO]{\thepage}
\renewcommand{\headrulewidth}{0mm}

\author{Guozheng Dai}
\address{Zhejiang University, Hangzhou, 310058,  China.}
\email{guozhengdai1997@gmail.com}

\author{Zhonggen Su}
\address{Zhejiang University, Hangzhou, 310058,  China.}
\email{suzhonggen@zju.edu.cn}

\author{Hanchao Wang }
\address{Shandong University,  Jinan,  250100, China.}
\email{hcwang06@gmail.com}

\date{}


\begin{abstract}	
	Let $M$ be an $n\times n$ random i.i.d. matrix. This paper studies the deviation inequality of $s_{n-k+1}(M)$, the $k$-th smallest singular value of $M$. In particular, when the entries of $M$ are subgaussian, we show that for any $\gamma\in (0, 1/2), \varepsilon>0$ and $\log n\le k\le c\sqrt{n}$
	\begin{align}
		\textsf{P}\{s_{n-k+1}(M)\le \frac{\varepsilon}{\sqrt{n}}  \}\le \Big( \frac{C\varepsilon}{k}\Big)^{\gamma k^{2}}+e^{-c_{1}kn}.\nonumber
	\end{align}
    This result improves an existing result of Nguyen, which obtained a deviation inequality of $s_{n-k+1}(M)$ with $(C\varepsilon/k)^{\gamma k^{2}}+e^{-cn}$ decay.	
\end{abstract}

\maketitle



\section{Introduction and main results}

Let $M=(\xi_{ij})_{n\times n}$ be a random matrix, where $\xi_{ij}$ are independent copies of a random variable $\xi$ in $\mathbb{R}$. The singular values $s_{k}(M)$ of $M$ are the eigenvalues of $\sqrt{M^{\top}M}$ arranged in the non-increasing order. By the Courant–Fischer–Weyl min-max principle, $s_{k}(M)$ can be formulated as follows:
\begin{align}\label{Eq_intro_def_singularvalues}
	s_{k}(M)=\min_{\text{dim}(H)=n-k+1}\max_{x\in H\cap S^{n-1}}\Vert Mx\Vert_{2},
\end{align}
where $H$ is a linear subspace of $\mathbb{R}^{n}$ with dimension $n-k+1$ and $S^{n-1}$ is the unit sphere.

The condition number of $M$, i.e., the ratio of $s_{1}(M)$ and $s_{n}(M)$, is used in numerical analysis as a measure of sensitivity to round-off errors. Bounding the condition number of $M$ is an important problem in this field. The  behavior of $s_{1}(M)$ is well understood under general assumptions on $\xi$. In particular, Yin, Bai and Krishnaiah \cite{Yin_PTRF} proved that, when $\textsf{E}\xi=0$ and $\textsf{E}\vert\xi\vert^{4}<\infty$, with high probability
\begin{align}
	s_{1}(M)\sim \sqrt{n}.\nonumber
\end{align}
Hence, the main problem in estimating the magnitude of the condition number is to explore the lower bound of $s_{n}(M)$.

In the case where $\xi$ is a standard gaussian variable, i.e., $M$ is a Ginibre matrix, Edelman \cite{Edelman_SIAM_JMAA} computed the distribution of $s_{n}(M)$. He showed for every $\varepsilon\ge 0$
\begin{align}
	\textsf{P}\{s_{n}(M)\le \varepsilon n^{-1/2}  \}\sim \varepsilon.\nonumber
\end{align}

Motivated by the universality phenomenon in the random matrix theory, we naturally expect similar probability inequalities  for $s_{n}(M)$ in general random matrix ensembles. Spielman and Teng \cite{Spielman_ICM} conjectured that when $\xi$ is a Rademacher variable, i.e., $\textsf{P}\{\xi=1 \}=\textsf{P}\{\xi=-1 \}=1/2$, for $\varepsilon\ge 0$
\begin{align}
	\textsf{P}\{s_{n}(M)\le \varepsilon n^{-1/2}  \}\le \varepsilon+e^{-cn},\nonumber
\end{align}
where $c$ is an absolute positive constant. This conjecture, put forward in the 2002 ICM, has stimulated a great deal of work on the least singular values of random i.i.d. matrices in the past 20 years. Compared with the gaussian case, the Spielman-Teng's conjecture has an extra exponential decay term. We note that the exponential addition is necessary for the deviation inequalities of $s_{n}(M)$ in this setting. Indeed, Kahn et al. \cite{Kahn_Jams} showed that
\begin{align}
	\textsf{P}\{s_{n}(M)=0  \}\le \big(0.998+o(1)  \big)^{n}.\nonumber
\end{align}
Recently, Tikhomirov \cite{Tikhomirov_AOM} obtained the asymptotically optimal exponent:
\begin{align}
	\textsf{P}\{s_{n}(M)=0  \}= \big(\frac{1}{2}+o(1)  \big)^{n}.\nonumber
\end{align}

Most notably Rudelson and Vershynin \cite{Rudelson_Advance} proved Spielman-Teng's conjecture up to a constant. In particular, they assumed $\xi$ is a  subgaussian variable with mean $0$, variance $1$, and showed for $\varepsilon\ge 0$
\begin{align}
	\textsf{P}\{s_{n}(M)\le \varepsilon n^{-1/2}  \}\le C\varepsilon+e^{-cn},\nonumber
\end{align}
where $C$ is an absolute positive constant depending only on the subgaussian moment of $\xi$. Here, we say $\xi$ is subgaussian if satisfying for $t\ge 0$
\begin{align}\label{Eq_Intro_Rudelson_Vershynin_result}
	\textsf{P}\{\vert\xi\vert>t  \}\le 2\exp(-t^{2}/K^{2}),
\end{align}
and  call the smallest of such $K$ the subgaussian moment of $\xi$. Recently, Rebrova, Tikhomirov \cite{Rebrova_IJM} and Livshyts \cite{Livshys_JAM} extended \eqref{Eq_Intro_Rudelson_Vershynin_result} to the case where $\xi$ is a heavy-tailed random variable. We refer interested readers to the corresponding work, and their results will not be detailed here.

Another famous result \cite{Tao_GAFA} was proved by Tao and Vu in the setting where $\xi$ is a  standard variable and satisfies $\textsf{E}\vert\xi\vert^C<\infty$ for some sufficiently large absolute constant $C$. They showed for all $\varepsilon\ge 0$
\begin{align}\label{Intro_Tao_universal}
		\textsf{P}\{s_{n}(M)\le \varepsilon n^{-1/2}  \}= \varepsilon+O(n^{-c}),
\end{align}
where $c>0$ is a absolute constant. Although \eqref{Intro_Tao_universal} captures the optimal constant $1$ before $\varepsilon$, it has a polynomial decay term which shall lead to a bad bound for the probability of the event $\{s_{n}(M)=0\}$.

Very recently, Sah, Sahasrabudhe, and Sawhney \cite{Sah_arXiv_conjecture} proved the Spielman-Teng conjecture up to a $1+o(1)$ factor. In particular, they showed that, when $\xi$ is a standard subgaussian variable,  for all $\varepsilon\ge 0$
\begin{align}
	\textsf{P}\{s_{n}(M)\le \varepsilon n^{-1/2}  \}= \big(1+C\log^{-1/16}n  \big)\varepsilon+e^{-cn},\nonumber
\end{align}
where $C, c$ are absolute constants depending only on the subgaussian moment of $\xi$. 

As we have seen above, our understanding of the lower bound of $s_{n}(A)$ is almost complete. So it is natural to ask how other singular values $s_{k}(M)$ behave. In the Ginibre ensemble, a celebrated result due to Szarek \cite{Szarek_JC} yields for all $\varepsilon\ge 0$ and $1\le k\le n$
\begin{align}
	\big(\frac{C_{1}\varepsilon}{k}\big)^{k^2}\le \textsf{P}\{s_{n-k+1}(M)\le \frac{\varepsilon}{\sqrt{n}}  \}\le \big(\frac{C_{2}\varepsilon}{k}  \big)^{k^2},\nonumber
\end{align}
where $C_{1}, C_{2}$ are absolute positive constants. An extended result was given by Nguyen \cite{Nguyen_JFA}. He considered the standard subgaussian case and proved that, for any $0<\gamma<1$, there exist positive constants $C, c$ and $\gamma_{0}$ (depending only on the subgaussian moment and $\gamma$) such that for $\gamma_{0}^{-1}<k<\gamma_{0}n$ and $\varepsilon\ge 0$
\begin{align}\label{Intro_Nguyen_result}
	\textsf{P}\{s_{n-k+1}(M)\le \frac{\varepsilon}{\sqrt{n}}  \}\le \big( \frac{C\varepsilon}{k}  \big)^{\gamma k^{2}}+e^{-cn}.
\end{align}


In the setting of \eqref{Intro_Nguyen_result}, i.e., $\xi$ is a standard subgaussian variable, we have by taking $\varepsilon=0$
\begin{align}
	\textsf{P}\{s_{n-k+1}(M)=0  \}\le e^{-cn}.\nonumber
\end{align}
Recently, Rudelson \cite{Rudelson_Annals} considered the probability that $M$ has a large co-rank. He showed that, when $\xi$ is a non-constant subgaussian variable, for $k\le c_{1}\sqrt{n}$
\begin{align}\label{Intro_Rudelson_rank_result}
	\textsf{P}\{ \text{rank}(M)\le n-k \}\le e^{-c_{2}kn}.
\end{align}
Note that 
\begin{align}
	\text{rank}(M)=\text{rank}(M^\top M)=\min\{k: s_{k}(M)>0, 1\le k\le n \}.\nonumber
\end{align}
Hence, \eqref{Intro_Rudelson_rank_result} yields for $k\le c_{1}\sqrt{n}$
\begin{align}
	\textsf{P}\{s_{n-k+1}(M)=0  \}\le e^{-c_{2}kn}.\nonumber
\end{align}
Due to this observation, one would expect the term $e^{-cn}$ in \eqref{Intro_Nguyen_result} can be replaced by $e^{-ckn}$. Our main result shows that this is indeed possible.
\begin{mytheo}\label{Theo_Main2}
	Let $M=(\xi_{ij})_{n\times n}$ be a random  matrix and $\xi_{ij}$ be independent copies of a random variable $\xi$ in $\mathbb{R}$. Assume $\xi$ is a centered non-constant subgaussian variable. For any fixed $\gamma\in (0, 1/2)$, we have for $\varepsilon>0$ and $\log n\le k\le c\sqrt{n}$
	\begin{align}
		\textsf{P}\{s_{n-k+1}(M)\le \frac{\varepsilon}{\sqrt{n}}  \}\le \Big( \frac{C\varepsilon}{k}\Big)^{\gamma k^{2}}+e^{-c_{1}kn},\nonumber
	\end{align}
where $C, c, c_{1}$ are positive constants depending only on $\gamma$ and the subgaussian moment of $\xi$.
\end{mytheo}



\textbf{Outline of the proof :}  Inspired by the work \cite{Nguyen_JFA} of Nguyen, we first translate the estimate of $s_{n-k+1}(M)$ to bounding the distance between the column of $M$ and an independent random subspace. In particular, consider a random vector $X\in\mathbb{R}^{n}$ with i.i.d. centered non-constant subgaussian coordinates. Let $X_{1}, \cdots, X_{n}\stackrel{i.i.d.}{\sim}X$  and  $H$ be a random space spanned by $X_{k+1}, \cdots, X_{n}$. We can translate the problem in Theorem \ref{Theo_Main2} to obtain a super-exponential probability bound of $\sum_{i\le k}\textnormal{dist}(X_{i}, H)$. 
In this step, the restricted invertibility (see Lemma \ref{Lem_Naor_restricted_invertibility} below) of a full rank matrix plays an important role.


The above distance estimate can be treated as the Littlewood-Offord problem, which has been well explored in \cite{Rudelson_Advance,Rudelson_CPAM}. In particular,  one can take a realization $\omega$ of $\Omega_{1}:=\{D^{(1)}_{c_{1}\sqrt{n}, c_{2}}(H^\perp)>C\sqrt{n}e^{\frac{c_{3}n}{k}}  \}$, where $D^{(1)}_{c_{1}\sqrt{n}, c_{2}}(H^\perp)$ is the LCD of $H^\perp$ (see Definition \ref{Def_lcd1} below). Then, the random variables $\textnormal{dist}(X_{i}, H(\omega))$ are independent and one can estimate $\textnormal{dist}(X_{i}, H(\omega))$ based on the existing small ball probability results. Due to the independence of $\textnormal{dist}(X_{i}, H(\omega))$,  a tensorization argument  yields a bound of $\sum_{i\le k}\textnormal{dist}(X_{i}, H(\omega))$. This is what Nguyen \cite{Nguyen_JFA} did when proving \eqref{Intro_Nguyen_result}. Note that the probability bound of $\Omega_{1}^{c}$ is $e^{-c_{4}n}$ (see Lemma \ref{Lem_prob_LCD_1} below), which is enough for \eqref{Intro_Nguyen_result} but  too large for the super-exponential decay in Theorem \ref{Theo_Main2}. Therefore, this method is not suitable for our result.

Inspired by Rudelson's recent work \cite{Rudelson_Annals}, we select a bigger event $\Omega_{2}$ such that the tail probability of $\Omega_{2}^{c}$  is $e^{-c_{5}kn}$.   Condition on the event $\Omega_{2}$, we have $D^{(2)}_{c_{1}\sqrt{n}, c_{2}}(\tilde{F})>C\sqrt{n}e^{\frac{c_{3}n}{k}}$, where $\tilde{F}$ is a subspace of $H^\perp$. Then one can finish the proof based on the existing small ball probability results and the tensorization. The most technical part of our argument is how to select the event $\Omega_{2}$. To do this, we start with a large event $\Omega_{3}$, and then use Lemma \ref{Lem_set_partition} to remove some subsets of $\Omega_{3}$. Lemmas \ref{Lem_Rudelson_subgaussian_compressible} and \ref{Lem_Rudelson_subgaussian_incompressible} ensure that the removed events have a super-exponential decay. We also remark that Lemmas \ref{Lem_Rudelson_subgaussian_compressible} and \ref{Lem_Rudelson_subgaussian_incompressible} are important for us and their proofs are based on a technique named random rounding, which is widely used in computer science and introduced in random matrix theory by Livshyts \cite{Livshys_JAM} (see \cite{Rudelson_Annals} for details).

\textbf{Organization of the rest paper :} Section 2 will introduce some notations and auxiliary lemmas. In Section 3, we shall prove Theorem \ref{Theo_Main2}. Section 4 consider the singular value when the entries of random matrices are not subgaussian. In the Appendix, we shall explain why random variables with finite second moments have bounded L\'{e}vy concentration functions, which is not a trivial conclusion.

\section{Notations and preliminary results.}

\subsection{Notations.}
We denote by $[n]$ the set of natural numbers from $1$ to $n$. For a fixed vector $x\in\mathbb{R}^{n}$, we denote by $\Vert x\Vert_{2}$ its Euclidean norm, i.e., $\Vert x\Vert_{2}=\sqrt{\sum_{i\in [n]}x_{i}^{2}}$. The unit sphere of the metric space $(\mathbb{R}^{n}, \Vert\cdot\Vert_{2})$ is denoted by $S^{n-1}$.

For a given matrix $A=(a_{ij})_{m\times n}$. We denote by the $\text{Row}_{i}(A)$ its $i$-th row and by $\text{Col}_{j}(A)$ its $j$-th column. For a fixed set $J\subset [n]$, $A_{J}$ denotes the matrix whose column vectors are $\text{Col}_{j}(A), j\in J$. Let $\Vert A\Vert_{F}$ be the Frobenius norm of the matrix $A$. In particular, $\Vert A\Vert_{F}=\sqrt{\sum a_{ij}^{2}}$. We also denote by $A^\top$ the transpose of $A$ and write $\text{Ker}(A)$ as the kernel of $A$, a subset of $\mathbb{R}^{n}$ containing vectors such that $\Vert Ax\Vert_{2}=0$.

Let $x\in\mathbb{R}^{n}$ and $\delta\in(0, 1)$.
We say $x$ is $\delta$-sparse if $\vert\text{supp}(x)\vert\le \delta n$ and denote by $\text{Sparse}(\delta n)$ the set of $\delta$-sparse vectors. Based on the distance between $\text{Sparse}(\delta n)$ and the unit vectors, we can decompose $S^{n-1}$ into two parts: compressible vectors and incompressible vectors. Their sets are written as follows:
\begin{itemize}
	\item $\text{Comp}(\delta, \rho)=\{x\in S^{n-1}:  \text{dist}(x, \text{Sparse}(\delta n))\le \rho \}$,
	\item $\text{Incomp}(\delta, \rho)=S^{n-1}\backslash\text{Comp}(\delta, \rho)$.
\end{itemize}

Throughout the paper, we denote by $C, C_{1}, c, c_{1},\cdots$ absolute constants. They are independent of any parameters unless otherwise stated and  their values may change from line to line.

\subsection{The Restricted invertibility phenomenon.}

In this subsection, we introduce a powerful result of non-random matrices. This result allows one to obtain a well-conditioned minor from a full rank rectangular matrix, which is called the restricted invertibility phenomenon.

\begin{mylem}[Theorem 6 in \cite{Naor_JTDM}]\label{Lem_Naor_restricted_invertibility}
	Assume that $A$ is a full rank matrix of size $k\times d$ with $k\le d$. Then for $1\le l\le k-1$, there exist $l$ different indices $i_{1}, \cdots, i_{l}$ such that the matrix $A_{\{i_{1},\cdots, i_{l} \}}$ with columns $\textnormal{Col}_{i_{1}}(A), \cdots, \textnormal{Col}_{i_{l}}(A)$ has the  non-zero smallest singular value $s_{l}(A_{\{i_{1},\cdots, i_{l}  \}})$ satisfying
	\begin{align}
		s_{l}(A_{\{i_{1},\cdots, i_{l}  \}})^{-1}\le C\min_{r\in\{l+1,\cdots, k \}}\sqrt{\frac{rd}{(r-l)\sum_{i=r}^{k}s_{i}^{2}(A)}}.\nonumber
	\end{align}
\end{mylem}

\subsection{Least common denominators and the small ball probability.}

The least common denominator (LCD) of a sequence of real numbers was initially introduced in \cite{Rudelson_Advance} and has been an important tool to estimate the small ball probability. Many different versions of LCD have been developed to prove quantitative estimates of the invertibility of various random matrices. We refer the interested readers to a survey \cite{Tikhomirov_ICM} for details. Here, we introduce two different versions of LCD, which have appeared in \cite{Rudelson_CPAM,Rudelson_GAFA}.

\begin{mydef}\label{Def_lcd1}
	Fix $\alpha>0$ and $\gamma\in(0, 1)$. For a vector $a\in\mathbb{R}^{m}$, the LCD is defined as 
	\begin{align}
		D^{(1)}_{\alpha, \gamma}(a):=\inf\{\theta>0: \textnormal{dist}(\theta a, \mathbb{Z}^{m})<\min(\gamma \Vert\theta a\Vert_{2}, \alpha  )  \}.\nonumber
	\end{align}
Further, the LCD of a subspace $H\subset\mathbb{R}^{m}$ is defined as
\begin{align}
	D^{(1)}_{\alpha, \gamma}(H):=\inf_{a\in S^{m-1}\cap H}D^{(1)}_{\alpha, \gamma}(a).\nonumber
\end{align}
Lastly, the LCD of an $m\times n$ matrix $A$ is defined as
\begin{align}
	D^{(1)}_{\alpha, \gamma}(A):=\inf\{\Vert\theta\Vert_{2}:\theta\in\mathbb{R}^{m}, \textnormal{dist}(A^\top\theta, \mathbb{Z}^{n})<\min(\gamma \Vert A^\top\theta\Vert_{2}, \alpha)  \}.\nonumber
\end{align}
\end{mydef}

The L\'{e}vy concentration function of a random vector $X\in\mathbb{R}^{m}$, which quantifies the spread of $X$, is defined as follows: 
\begin{align}
	\mathcal{L}(X, t)=\sup_{y\in\mathbb{R}^{m}}\textsf{P}\big\{\Vert X-y\Vert_{2}\le t  \big\}.\nonumber
\end{align}
This function  measures the small ball probabilities, i.e., the likelihood that the random vector $X$ enters a small ball in the space.
A random vector $X$ is said to have a bounded L\'{e}vy concentration function if there exist absolute constants $a>0$ and $b\in(0, 1)$ such that $\mathcal{L}(X, a)<b$. 

The following result shows a bound for the distance between a random vector  and an arbitrary fixed subspace. This bound will depend on the structure of the subspace, which is expressed by the LCD.

\begin{mylem}[Corollary 4.8 in \cite{Livshys_JAM}]\label{Lem_distance_to_general_subspace_random_rounding}
	Let $X$ be a random vector in $\mathbb{R}^{n}$ whose coordinates are independent copies of $\xi$. Assume that there exist $a, b>0$ such that $\mathcal{L}(\xi, a)<b$. Let $m<n$ and $H$ be a fixed $(n-m)$-dimensional subspace of $\mathbb{R}^{n}$. Then for any $v\in\mathbb{R}^{n}, \alpha>0$ and $\gamma\in (0, 1)$, one has for every $\varepsilon>\sqrt{n}/D^{(1)}_{\alpha, \gamma}(H^{\perp})$
	\begin{align}
		\textsf{P}\{\textnormal{dist}(X, H+v)\le \varepsilon \sqrt{m}   \}\le \big( \frac{C_{1}\varepsilon}{\gamma}  \big)^{m}+C_{2}^{m}e^{-c\alpha^{2}},\nonumber
	\end{align}
where $C_{1}, C_{2}$ and $c$ are positive constants depending only on $a, b$.
\end{mylem}

From Lemma \ref{Lem_distance_to_general_subspace_random_rounding}, we can see that to  describe better the behavior of the quantity $\text{dist}(X, H+v)$, we need to estimate the $D^{(1)}_{\alpha, \gamma}(H^{\perp})$. The following lemma shows that the LCD of a random subspace is exponentially large.

\begin{mylem}[Theorem 6.2 in \cite{Livshys_JAM}]\label{Lem_prob_LCD_1}
	Let $X$ be a random vector in $\mathbb{R}^{n}$ with i.i.d. coordinates $\xi_{i}$. Assume that there exist $a, b, K>0$ such that
	\begin{align}
		\mathcal{L}(\xi_{1}, a)<b, \quad \textsf{E}\Vert X\Vert_{2}\le K\sqrt{n}.\nonumber
	\end{align}
For a fixed $m$, let $M$ be an $n\times (n-m)$ matrix whose columns are independent copies of $X$. Let $H^\perp=\textnormal{Ker}(M^\top)$ be the orthogonal complement space of $\textnormal{span}\{\textnormal{Col}_{i}(M), 1\le i\le n-m  \}$. Then for an appropriate constant $c>0$, depending only on $a, b$ and $K$, we have for $1\le m\le cn$ 
\begin{align}
	\textsf{P}\{D^{(1)}_{c_{1}\sqrt{n}, c_{2}}(H^\perp)\le C\sqrt{n}e^{\frac{c_{3}n}{m}}   \}\le e^{-c_{4}n},\nonumber
\end{align}
where $C, c_{1},\cdots, c_{4}$ are positive constants depending only on $a, b, K$.
\end{mylem}

Next, we shall introduce another version of LCD. 
\begin{mydef}\label{Def_lcd2}
	Let $A$ be an $m\times n$ matrix, and let $\alpha>0, \gamma\in(0, 1)$. Define the LCD of $A$ as follows:
	\begin{align}
		D^{(2)}_{\alpha, \gamma}(A):=\inf\Big(\Vert \theta\Vert_{2}: \theta\in\mathbb{R}^{m}, \textnormal{dist}(A^\top\theta, \mathbb{Z}^{n})<\alpha\sqrt{\log_{+}\frac{\gamma\Vert A^{\top}\theta\Vert_{2}}{\alpha}}   \Big).\nonumber
	\end{align}
Let $H\subset\mathbb{R}^{n}$ be a  subspace. Define the LCD of $H$ by
\begin{align}
		D^{(2)}_{\alpha, \gamma}(H):=D^{(2)}_{\alpha, \gamma}(P_{H})=\inf\Big(\Vert y\Vert_{2}: y\in H, \textnormal{dist}(y, \mathbb{Z}^{n})<\alpha\sqrt{\log_{+}\frac{\gamma\Vert y\Vert_{2}}{\alpha}}   \Big),\nonumber
\end{align}
where $P_{H}$ is the orthogonal projection on $H$.
\end{mydef}

\begin{mylem}[Corollary 3.10 in \cite{Rudelson_Annals}]\label{Lem_Rudelson_LCD_distance}
	Consider a random vector $X=(\xi_{1},\cdots,  \xi_{n})$, where $\xi_{k}$ are independent copies of a non-constant subgaussian random variable $\xi$ in $\mathbb{R}$. Let $E$ be a subspace of $\mathbb{R}^{n}$ with $\textnormal{dim}(E)=m$, and let $P_{E}$ be the orthogonal projection onto $E$. Then for every $\alpha\ge c\sqrt{m}$ we have
	\begin{align}
		\mathcal{L}(P_{E}X, t\sqrt{m})\le \big(\frac{C\alpha}{\gamma\sqrt{m}}   \big)^{m}\big(t+\frac{\sqrt{m}}{D^{(2)}_{\alpha, \gamma}(E)}  \big)^{m},\quad t\ge 0,\nonumber
	\end{align}
where $C, c$ are positive constants depending only on the subgaussian moment of $\xi$.
\end{mylem}

\begin{mylem}[Lemma 3.11 in \cite{Rudelson_Annals}]\label{Lem_Rudelson_annals_U_theta}
	Let $\delta, \rho\in (0, 1)$. Let $U$ be an $n\times l$ matrix satisfying 
	\begin{align}
		U\mathbb{R}^{l}\cap S^{n-1}\subset \textnormal{Incomp}(\delta, \rho).\nonumber
	\end{align}
Then for any $\theta=(\theta_{1}, \cdots, \theta_{l})$ and $\alpha>0$,
\begin{align}
	\textnormal{dist}(U\theta, \mathbb{Z}^{n})\ge \alpha\sqrt{\log_{+}\frac{\rho\Vert U\theta\Vert_{2}}{\alpha}}\,\,\,\text{when}\,\,\, \Vert U\theta\Vert_{2}\le \sqrt{\delta n}/2.\nonumber
\end{align}
\end{mylem}

\begin{mylem}[Theorem 1.1 in \cite{Rudelson_IMRN} ]\label{Lem_smallballprobability_bounded_density}
	Let $X=(\xi_{1}, \cdots, \xi_{n})^\top\in\mathbb{R}^{n}$ where $\xi_{i}$ are independent random variables. Assume that the density functions of $\xi_{i}$  are bounded by $K$ almost everywhere. Denote $P$ the orthogonal projection in $\mathbb{R}^{n}$ onto a $d$-dimensional subspace. Then the density function of $PX$ is bounded by $(CK)^{d}$ almost everywhere. Here $C$ is a positive absolute constant.
\end{mylem}

\subsection{Almost orthogonal vectors.}

Let $\nu\in(0, 1)$. An $r$-tuple of vectors $$(v_{1}, \cdots, v_{r})\subset \mathbb{R}^{n} \backslash \{0 \}$$ is called $\nu$-almost orthogonal if the $n\times r$ matrix $W=(\frac{v_{1}}{\Vert v_{1}\Vert_{2}}, \cdots, \frac{v_{r}}{\Vert v_{r}\Vert_{2}})$ satisfies
\begin{align}
	1-\nu\le s_{r}(W)\le s_{1}(W)\le 1+\nu.\nonumber
\end{align}

\begin{mylem}[Lemma 3.3 in \cite{Rudelson_Annals}]\label{Lem_set_partition}
	Let $W\subset \mathbb{R}^{n}\backslash\{ 0\}$ be a closed set and $r<l\le n$. Let $E\subset\mathbb{R}^{n}$ be a linear subspace satisfying $\textnormal{dim}(E)=l$. Then at least one of the following events occurs.
	
	\begin{itemize}
		\item[(1)]  There exist vectors $v_{1}, \cdots, v_{r}\in E\cap W$ satisfying
		
		$(\textnormal{i})$ The $r$-tuple $(v_{1}, \cdots, v_{r})$ is $(\frac{1}{8})$-almost orthogonal;
		
		$(\textnormal{ii})$ For any $\theta=(\theta_{1}, \cdots, \theta_{r})^\top$, we have $$\sum\theta_{i}v_{i}\notin W\,\,\, \text{when}\,\,\, \Vert\theta\Vert_{2}\le \frac{1}{20\sqrt{r}}.$$
		
		\item[(2)] There exists a subspace $F\subset E$ with $\textnormal{dim}(F)=l-r$ such that $F\cap W=\emptyset$.

	\end{itemize}

\end{mylem}

The next result shows that it is unlikely that the kernel of a random matrix with i.i.d. centered non-constant subgaussian variables contains a large almost orthogonal system of compressible vectors.

\begin{mylem}[Proposition 4.2 in \cite{Rudelson_Annals}]\label{Lem_Rudelson_subgaussian_compressible}
	Let $l, n$ be positive integers such that $l\le n/2$ and let $M$ be an $(n-l)\times n$ matrix whose entries are i.i.d. centered non-constant subgaussian variables $\xi_{ij}$. There exsits $\tau >0$ such that the probability that there exists a $\frac{1}{4}$-almost orthogonal $r$-tuple $v_{1}, \cdots, v_{r}\in\textnormal{Comp}(\tau^{2}, \tau^{4})$ with $r\le \tau^{3}n$ and 
	\begin{align}
		\Vert Mv_{j}\Vert_{2}\le \tau \sqrt{n}, \quad\forall j\in [r]\nonumber
	\end{align}
is less than $e^{-crn}$. Here, the positive constants $\tau, c$ appeared above depend only on the subgaussian moment of $\xi_{11}$.
\end{mylem}
The next result  bounds the probability that the kernel of a random matrix with i.i.d. centered non-constant subgaussian variables contains an almost orthogonal system of incompressible vectors.

\begin{mylem}[Proposition 5.1 in \cite{Rudelson_Annals}]\label{Lem_Rudelson_subgaussian_incompressible}
	Let $M$ be an $(n-l)\times n$ matrix with i.i.d. centered non-constant subgaussian entries $\xi_{ij}$. Let $\rho=\rho(\tau)$ be a positive constant, where $\tau$ is a constant appeared in Lemma \ref{Lem_Rudelson_subgaussian_compressible}. Assume that $r\le l\le \rho\sqrt{n}/2$.
	
	Consider the event $\mathcal{E}_{r}$ that there exist vectors $v_{1}, \cdots, v_{r}\in\textnormal{Ker}(W)$ having the following properties:
	\begin{itemize}
		\item $\frac{\tau}{8}\sqrt{n}\le \Vert v_{j}\Vert_{2}\le \exp\big(\frac{c\rho^{2}n}{l}   \big)$ for all $j\in [r]$, where $c>0$ depending only on the subgaussian moment of $\xi_{11}$;
		\item $\textnormal{span}\{v_{1}, \cdots, v_{r}\}\cap S^{n-1}\subset \textnormal{Incomp}(\tau^{2}, \tau^{4})$;
		\item The vectors $v_{1}, \cdots, v_{r}$ are $\frac{1}{8}$-almost orthogonal;
		\item $\textnormal{dist}(v_{j}, \mathbb{Z}^{n})\le \rho\sqrt{n}$ for $j\in [r]$;
		\item The $n\times r$ matrix $V=(v_{1}, \cdots, v_{r})$ satisfies 
		\begin{align}
			\textnormal{dist}(V\theta, \mathbb{Z}^{n})>\rho\sqrt{n}\nonumber
		\end{align}
	for all $\theta\in\mathbb{R}^{r}$ such that $\Vert \theta\Vert_{2}\le \frac{1}{20\sqrt{r}}$ and $\Vert V\theta\Vert_{2}\ge \frac{\tau}{8}\sqrt{n}$.
	\end{itemize}
Then we have $\textsf{P}\{\mathcal{E}_{r}  \}\le e^{-rn}.$
\end{mylem}

\subsection{The smallest singular value }
In this subsection, we shall introduce a deviation inequality of $s_{n}(M)$, where $M$ is a random matrix as in Theorem \ref{Theo_Main1}.
\begin{mylem}[Corollary 2 in \cite{Livshys_JAM}]\label{Lem_smallest_singular_heavy_tailed_entries}
	Let $M=(\xi_{ij})$ be an $n\times n$ random matrix, where $\xi_{ij}$ are i.i.d. variables satisfying the condition \eqref{condition_in_Theo1}. Then for every $\varepsilon>0$, we have
	\begin{align}
		\textsf{P}\big\{s_{n}(M)\le  \frac{\varepsilon}{\sqrt{n}}  \big\}\le C\varepsilon+e^{-cn},\nonumber
	\end{align}
where $C$ and $c$ are positive constants depending only on parameters $a, b$ and $K$ in Theorem \ref{Theo_Main1}.
\end{mylem}

\section{Proof of our main result.}

\begin{proof}[Proof of Theorem \ref{Theo_Main2}]

We only need to prove Theorem \ref{Theo_Main2} when $n$ is large enough. In particular, assume that Theorem \ref{Theo_Main2} is valid when $n\ge n_{0}$. Note that
\begin{align}
	\lim_{\varepsilon\to 0}\textsf{P}\big\{s_{n-k+1}(M)\le \frac{\varepsilon}{\sqrt{n}}  \big\}=\textsf{P}\big\{s_{n-k+1}(M)=0  \big\}\le e^{-ckn},\nonumber
\end{align}
where the last inequality is due to \eqref{Intro_Rudelson_rank_result}. Hence, for any fixed $n$, there exists an $\varepsilon(n, n_{0})$ such that, when $\varepsilon\le \varepsilon(n, n_{0})$
\begin{align}
	\textsf{P}\big\{s_{n-k+1}(M)\le \frac{\varepsilon}{\sqrt{n}}  \big\}-\textsf{P}\big\{s_{n-k+1}(M)=0  \big\}\le e^{-cn_{0}^{2}}.\nonumber
\end{align}
Let $\varepsilon(n_{0})=\min\{\varepsilon(1, n_{0}), \cdots,  \varepsilon(n_{0}-1, n_{0})  \}$. Then we have for $n<n_{0}$ and $\varepsilon\le \varepsilon(n_{0})$
\begin{align}
	\textsf{P}\big\{s_{n-k+1}(M)\le \frac{\varepsilon}{\sqrt{n}}  \big\}\le 2e^{-ckn}.\nonumber
\end{align}
For the case $n<n_{0}$ and $\varepsilon>\varepsilon(n_{0})$, the desired bound is trivial.
Hence, it is enough to prove Theorem \ref{Theo_Main2} when $n$ is large enough.

Recall the min-max principle \eqref{Eq_intro_def_singularvalues}. So if $s_{n-k+1}(M)\le \frac{\varepsilon}{\sqrt{n}}$, then there exist $k$ orthogonal unit vectors $z_{1}, \cdots, z_{k}$ such that
\begin{align}
	\Vert Mz_{i}\Vert_{2} \le \frac{\varepsilon}{\sqrt{n}},\quad 1\le i\le k.\nonumber
\end{align} 

Let $Z^{\top}=(z_{1},\cdots, z_{k})$ be the $n\times k$ matrix generated by the vectors $z_{i}$. By virtue of Lemma \ref{Lem_Naor_restricted_invertibility}, we can extract from $Z$ a well-conditioned minor. In particular, for any $1\le l\le k-1$, Lemma \ref{Lem_Naor_restricted_invertibility} yields that  there exist $l$ distinct indices $i_{1}, \cdots, i_{l}\in [n]$ such 
that 
\begin{align}
	s_{l}(Z_{\{i_{1},\cdots, i_{l}\}})^{-1}&\le C_{1}\min_{r\in \{l+1,\cdots, k \}}\sqrt{\frac{rn}{(r-l)\sum_{i=r}^{k}s_{i}^{2}(Z)}}\nonumber\\
	&\le C_{1} \min_{r\in \{l+1,\cdots, k \}}\sqrt{\frac{rn}{(r-l)(k-r+1)}},\nonumber
\end{align}
where the second inequality is due to that $ZZ^\top$ is an  indentity matrix. Hence, we have by taking $r=\lceil (k+l)/2\rceil$
\begin{align}\label{Eq_proof1_boundfors_{l}}
	s_{l}(Z_{\{i_{1},\cdots, i_{l}\}})^{-1}\le C_{1}\sqrt{\frac{kn}{(k-l)^{2}}}.
\end{align}

Let $A$ be the $l\times k$ matrix $\big( Z_{\{i_{1}, \cdots, i_{l}\}}  \big)^{\top}$ and $A^\prime=\big( Z_{\{i_{l+1}, \cdots, i_{n}\}}  \big)^{\top}$. For short, let $\boldsymbol{c}_{i}:=\text{Col}_{i}(M)$ be the $i$-th column of $M$. Then
\begin{align}
	B:=MZ^{\top}=(\boldsymbol{c}_{i_{1}},\cdots, \boldsymbol{c}_{i_{l}})A+(\boldsymbol{c}_{i_{l+1}},\cdots, \boldsymbol{c}_{i_{n}})A^{\prime}.\nonumber
\end{align}
Let $\hat{A}=A^\top(AA^\top)^{-1}$ be the right inverse matrix of $A$, i.e. $A\hat{A}$ is an identity matrix. Hence, we have
\begin{align}
	B\hat{A}=(\boldsymbol{c}_{i_{1}}, \cdots, \boldsymbol{c}_{i_{l}})+(\boldsymbol{c}_{i_{l+1}}, \cdots, \boldsymbol{c}_{i_{n}})A^\prime \hat{A}.\nonumber
\end{align}
Recalling the definitions of the largest and smallest singular values (see \eqref{Eq_intro_def_singularvalues}), we have for any $x\in S^{l-1}$
\begin{align}
	1=s_{1}(A\hat{A})\ge \Vert A\hat{A}x\Vert_{2}\ge s_{l}(A)\Vert \hat{A}x\Vert_{2}.\nonumber
\end{align}
Hence, we have by \eqref{Eq_proof1_boundfors_{l}}
\begin{align}
	s_{1}(\hat{A})\le s_{l}(A)^{-1}\le C_{1}\sqrt{\frac{kn}{(k-l)^{2}}}.\nonumber
\end{align}
Note that $$\Vert\text{Col}_{i}(B)\Vert_{2}\le \varepsilon/\sqrt{n},\quad 1\le i\le k.$$ Then, we have
\begin{align}
	\Vert B\hat{A}\Vert_{F}\le s_{1}(\hat{A})\Vert B\Vert_{F}\le C_{1}\sqrt{\frac{kn}{(k-l)^{2}}}\sqrt{\frac{k\varepsilon^{2}}{n}}\le C_{1}\frac{k\varepsilon}{k-l}.\nonumber
\end{align}

Let $H$ be the linear space spanned by $\boldsymbol{c}_{i_{l+1}}, \cdots, \boldsymbol{c}_{i_{n}}$. Let $P_{H^\perp}$ be the orthogonal projection in $\mathbb{R}^{n}$ onto $H^\perp$,  the orthogonal complement sapce of $H$. Then we have
\begin{align}
	P_{H^\perp}B\hat{A}=P_{H^\perp}(\boldsymbol{c}_{i_1}, \cdots, \boldsymbol{c}_{i_l}),\nonumber
\end{align}
which implies that
\begin{align}
	\Vert P_{H^{\perp}}\boldsymbol{c}_{i_{1}}\Vert_{2}^{2}+\cdots+\Vert P_{H^{\perp}}\boldsymbol{c}_{i_{l}}\Vert_{2}^{2}=\Vert P_{H^\perp}B\hat{A}\Vert_{F}^{2}\le \Vert B\hat{A}\Vert^{2}_{F}
	\le \Big( \frac{C_{1}k\varepsilon}{k-l}  \Big)^{2}.\nonumber
\end{align}

Up to now, we have proved for
$\varepsilon>0$ and $1\le l\le k-1$
\begin{align}\label{Eq_Proof2_trans}
	\textsf{P}\{s_{n-k+1}(M)\le \frac{\varepsilon}{\sqrt{n}}  \}\le \binom{n}{k}\binom{n}{l}\textsf{P}\Big\{ \sum_{m=1}^{l}\Vert P_{H^{\perp}}\boldsymbol{c}_{i_{m}}\Vert_{2}^{2}\le \Big(\frac{C_{1}k\varepsilon}{k-l}\Big)^{2}  \Big\}.
\end{align}

Let $\tau$ be the constant as in Lemma \ref{Lem_Rudelson_subgaussian_compressible} and denote
\begin{align}
	W_{0}=\textnormal{Comp}(\tau^{2}, \tau^{4}).\nonumber
\end{align}
Then, using Lemma \ref{Lem_set_partition} with $W=W_{0}, E=H^\perp$ and $r=l/4$, we obtain that at least one of the events described in (1) and (2) of Lemma \ref{Lem_set_partition} holds. We denote these events $\mathcal{E}_{1}$ and $\mathcal{E}_{2}$ respectively.

When $\mathcal{E}_{1}$ occurs, there exist $l/4$ $\big(1/8\big)$-almost orthogonal vectors in $W_{0}\cap H^\perp$. Hence, we have by Lemma \ref{Lem_Rudelson_subgaussian_compressible}
\begin{align}
	\textsf{P}\{\mathcal{E}_{1}  \}\le e^{-\frac{cln}{4}}.\nonumber
\end{align}

Assume that $\mathcal{E}_{2}$ holds. Consider the subspace $F\subset H^\perp$ such that 
\begin{align}
	\textnormal{dim}(F)=\frac{3}{4}l,\quad F\cap W_{0}=\emptyset.\nonumber
\end{align}
Let $\rho$ be the constant as in Lemma \ref{Lem_Rudelson_subgaussian_incompressible}. Define the following set:
\begin{align}
	W_{1}=\Big\{v\in F: \frac{\tau}{8}\sqrt{n}\le \Vert v\Vert_{2}\le \exp\big(\frac{c\rho^{2}n}{l} \big),\quad  \textnormal{dist}(v, \mathbb{Z}^{n})\le\rho \sqrt{n}  \Big\}.\nonumber
\end{align}
Applying Lemma \ref{Lem_set_partition} with $W=W_{1}, E=F$ and $r=l/4$, we  conclude that at least one of the following events holds:
\begin{itemize}
	\item[(1)] There exist vectors $v_{1}, \cdots, v_{l/4}\in F\cap W_{1}$ such that
	
	(i) $v_{1}, \cdots, v_{l/4}$ are $\big( \frac{1}{8} \big)$-almost orthogonal;
	
	(ii) For any $\theta=(\theta_{1},\cdots, \theta_{l/4})$, 
	\begin{align}
		\sum_{i=1}^{l/4}\theta_{i}v_{i}\notin W_{1},\,\,\,\text{when}\,\,\, \Vert\theta\Vert_{2}\le \frac{1}{20\sqrt{l/4}}.\nonumber
	\end{align} 
	
	\item[(2)]    There exists a subspace $\tilde{F}\subset F$ with $\textnormal{dim}(\tilde{F})=l/2$ satisfying $\tilde{F}\cap W_{1}=\emptyset$.
\end{itemize}
Denote the above events $\mathcal{V}_{1}$ and $\mathcal{V}_{2}$ respectively. Lemma \ref{Lem_Rudelson_subgaussian_incompressible} yields that
\begin{align}
	\textsf{P}\{\mathcal{V}_{1}  \}\le \exp(-\frac{ln}{4}).\nonumber
\end{align}

We next show that if the event $\mathcal{V}_{2}$ holds, 
\begin{align}
	D_{c_{1}\sqrt{l}, \tau^{4}}^{(2)}(\tilde{F})\ge \exp\big(\frac{c\rho^{2}n}{l}  \big),\nonumber
\end{align}
where $c_{1}$ is a positive constant determined in \eqref{Eq_proof2_c1} below.
Denote $S: \mathbb{R}^{l/2}\to \mathbb{R}^{n}$ an isometric embedding such that $S\mathbb{R}^{l/2}=\tilde {F}$. Then we have by Definition \ref{Def_lcd2}
\begin{align}
	D_{c_{1}\sqrt{l}, \tau^{4}}^{(2)}(\tilde{F})=D_{c_{1}\sqrt{l}, \tau^{4}}^{(2)}(S^{\top}).\nonumber
\end{align}
Let $\theta\in\mathbb{R}^{l/2}$ be a vector satisfying
\begin{align}
	\textnormal{dist}(S\theta, \mathbb{Z}^{n})<c_{1}\sqrt{l}\sqrt{\log_{+}\frac{\tau^{4}\Vert \theta\Vert_{2}}{c_{1}\sqrt{l}}}.\nonumber
\end{align}
Note that
\begin{align}
	S\mathbb{R}^{l/2}\cap S^{n-1}\subset F\cap S^{n-1}\subset \textnormal{Incomp}(\tau^{2}, \tau^{4}).\nonumber
\end{align}
Hence, Lemma \ref{Lem_Rudelson_annals_U_theta} yields that
\begin{align}
	\Vert \theta\Vert_{2}=\Vert U\theta\Vert_{2}\ge \frac{\tau\sqrt{n}}{2}.\nonumber
\end{align}
On the other hand, $\Vert \theta\Vert_{2}\le \exp\big(\frac{c\rho^{2}n}{l}  \big)$ implies that, with a proper constant $c_{1}$ and when $n$ is large enough,
\begin{align}\label{Eq_proof2_c1}
	c_{1}\sqrt{l}\sqrt{\log_{+}\frac{\tau^{4}\Vert \theta\Vert_{2}}{c_{1}\sqrt{l}}}\le c_{1}\sqrt{l}\sqrt{\frac{2c\rho^{2}n}{l}}\le \rho\sqrt{n},
\end{align}
and therefore $\textnormal{dist}(S\theta, \mathbb{R}^{n})<\rho\sqrt{n}$. The fact $S\mathbb{R}^{l/2}\cap W_{1}=\emptyset$ yields that
\begin{align}
	\Vert \theta\Vert_{2}=\Vert S\theta\Vert_{2}>\exp\big( \frac{c\rho^{2}n}{l} \big),
\end{align}
implying that
\begin{align}
	D_{c_{1}\sqrt{l}, \tau^{4}}^{(2)}(\tilde{F})\ge \exp\big(\frac{c\rho^{2}n}{l}  \big).\nonumber
\end{align}

Hence, we have 
\begin{align}\label{Eq_proof2_juli}
	&\textsf{P}\Big\{\sum_{m=1}^{l}\Vert P_{H^{\perp}}\boldsymbol{c}_{i_{m}}\Vert_{2}^{2}\le \Big(\frac{C_{1}k\varepsilon}{k-l}\Big)^{2}  \Big\}\nonumber\\
	\le& \textsf{P}\Big\{\sum_{m=1}^{l}\Vert P_{H^{\perp}}\boldsymbol{c}_{i_{m}}\Vert_{2}^{2}\le \Big(\frac{C_{1}k\varepsilon}{k-l}\Big)^{2}, \mathcal{V}_{2}  \Big\}+\textsf{P}\{\mathcal{E}_{1}  \}+\textsf{P}\{\mathcal{V}_{1}  \}\nonumber\\
	\le &\binom{n}{l/2}\textsf{P}\Big\{\sum_{m=1}^{l}\Vert P_{\tilde{F}}\boldsymbol{c}_{i_{m}}\Vert_{2}^{2}\le \Big(\frac{C_{1}k\varepsilon}{k-l}\Big)^{2}, \mathcal{V}_{2}  \Big\}+2\exp(-c_{2}ln),
\end{align}
where $\tilde{F}$ is the subspace of $H^\perp$ described in $\mathcal{V}_{2}$.

Condition on a realization $\omega$ of $\mathcal{V}_{2}$. Lemma \ref{Lem_Rudelson_LCD_distance} yields for $t\ge0$
\begin{align}\label{123}
	\textsf{P}\Big\{\Vert P_{\tilde{F}(\omega)}\boldsymbol{c}_{i_{m}}\Vert_{2}\le t\sqrt{l/2} \Big\}\le (C_{2}t)^{l/2}+\exp(-c_{3}n).
\end{align}
Note that the random variables $\Vert P_{\tilde{F}(\omega)}\boldsymbol{c}_{i_{1}}\Vert_{2},\cdots, \Vert P_{\tilde{F}(\omega)}\boldsymbol{c}_{i_{l}}\Vert_{2}$ are independent.
 We have for $t> 0$
\begin{align}
	\textsf{P}\Big\{ \sum_{m=1}^{l}\Vert P_{\tilde{F}(\omega)}\boldsymbol{c}_{i_{m}}\Vert_{2}^{2}\le l^{2}t^2  \Big\}&=\textsf{P}\Big\{l^{2}-\frac{1}{t^{2}} \sum_{m=1}^{l}\Vert P_{\tilde{F}(\omega)}\boldsymbol{c}_{i_{m}}\Vert_{2}^{2}\ge 0  \Big\}\nonumber\\
	&\le \textsf{E}\exp\Big(l^{2}-\frac{1}{t^{2}} \sum_{m=1}^{l}\Vert P_{\tilde{F}(\omega)}\boldsymbol{c}_{i_{m}}\Vert_{2}^{2}   \Big).\nonumber\\
	&=e^{l^{2}}\prod_{m=1}^{l}\textsf{E}\exp\Big(-\frac{\Vert P_{\tilde{F}(\omega)}\boldsymbol{c}_{i_{m}}\Vert_{2}^{2}}{t^{2}}  \Big).\nonumber
\end{align}
Note that 
\begin{align}
	\textsf{E}\exp\Big(-\frac{\Vert P_{\tilde{F}(\omega)}\boldsymbol{c}_{i_{m}}\Vert_{2}^{2}}{t^{2}}  \Big)&=\int_{0}^{1}\textsf{P}\big\{ \exp\big(-\frac{\Vert P_{\tilde{F}(\omega)}\boldsymbol{c}_{i_{m}}\Vert_{2}^{2}}{t^{2}}  \big)>s     \big\}\,ds\nonumber\\
	&=\int_{0}^{\infty}2ue^{-u^{2}}\textsf{P}\{ \Vert P_{\tilde{F}(\omega)}\boldsymbol{c}_{i_{m}}\Vert_{2}<tu  \}\,du.\nonumber
\end{align}
Hence, we have by \eqref{123}
\begin{align}
	\textsf{E}\exp\Big(-\frac{\Vert P_{\tilde{F}(\omega)}\boldsymbol{c}_{i_{m}}\Vert_{2}^{2}}{t^{2}}  \Big)\le \big(  \frac{C_{3}t}{\sqrt{l}}  \big)^{l/2}\Gamma(\frac{l}{4}+1)+e^{-c_{4}n}\le (C_{4}t)^{l/2}+e^{-c_{4}n}, \nonumber
\end{align}
where $\Gamma(\cdot)$ is the Gamma function. Hence, we have for $t>0$ and $n$ is large enough
\begin{align}
	\textsf{P}\Big\{ \sum_{m=1}^{l}\Vert P_{\tilde{F}(\omega)}\boldsymbol{c}_{i_{m}}\Vert_{2}^{2}\le l^{2}t^2  \Big\}\le (C_{5}t)^{l^2/2}+e^{-c_{6}nl}.\nonumber
\end{align}
Then, we have for $\varepsilon\ge 0$
\begin{align}
	\textsf{P}\Big\{\sum_{m=1}^{l}\Vert P_{\tilde{F}}\boldsymbol{c}_{i_{m}}\Vert_{2}^{2}\le \Big(\frac{C_{1}k\varepsilon}{k-l}\Big)^{2},\mathcal{V}_{2}  \Big\}\le \big( \frac{C_{6}k\varepsilon}{l(k-l)}\big)^{l^{2}/2}+e^{-c_{6}nl}.\nonumber
\end{align}
By virtue of \eqref{Eq_Proof2_trans} and \eqref{Eq_proof2_juli}, we have for $\varepsilon\ge 0$ and $1\le l\le k-1$
\begin{align}
	\textsf{P}\{s_{n-k+1}(M)\le \frac{\varepsilon}{\sqrt{n}}  \}&\le \binom{n}{k}\binom{n}{l}\binom{n}{l/2}\Big(\big( \frac{C_{7}k\varepsilon}{l(k-l)}\big)^{l^{2}/2}+e^{-c_{6}nl}  \Big)\nonumber\\
	&\le n^{k+l+l/2}\Big(\big( \frac{C_{7}k\varepsilon}{l(k-l)}\big)^{l^{2}/2}+e^{-c_{6}nl} \Big).\nonumber
\end{align}
Note that $\log n\le k\le c\sqrt{n}, \alpha\in (0, 1)$.
Hence we have 
\begin{align}
	n^{1/k}=O(1),\quad n^{k}=o(e^{n}).\nonumber
\end{align}
Let $l=\sqrt{2\gamma}k, \gamma\in(0, 1/2)$ ($l\le k-1$ when $n$ is large enough). We have by adjusting the absolute constants
\begin{align}
	\textsf{P}\{s_{n-k+1}(M)\le \frac{\varepsilon}{\sqrt{n}}  \}\le \big( \frac{C_{8}(\gamma)\varepsilon}{k}\big)^{\gamma k^{2}}+e^{-c_{7}nl},\nonumber
\end{align}
as desired.

\end{proof}

\section{Singular values in non-subgaussian case}
Let $M=(\xi_{ij})_{n\times n}$ be a random matrix and $\xi_{ij}$ be independent copies of $\xi$. In this section, we consider the singular value of $M$ when $\xi$ is not a subgaussian variable. In particular, we consider two cases here: (i). $\xi$ has a bounded L\'{e}vy concentration function; (ii). the density function of $\xi$ is bounded.

\subsection{Bounded L\'{e}vy concentration function case}

\begin{mytheo}\label{Theo_Main1}
	Let $M=(\xi_{ij})_{n\times n}$ be a random  matrix and $\xi_{ij}$ be independent copies of a random variable $\xi$ in $\mathbb{R}$. Assume that
	\begin{align}\label{condition_in_Theo1}
		\mathcal{L}(\xi, a)< b,\quad \textsf{E}\sqrt{\sum_{i\le n}\xi_{i1}^{2}}\le K\sqrt{n},
	\end{align}
	where $b\in(0, 1)$ and $a, K>0$. For any fixed $\alpha, \gamma\in (0, 1)$, we have for $\varepsilon>0$ and $\log n\le k\le n^{\alpha}$
	\begin{align}
		\textsf{P}\{s_{n-k+1}(M)\le \frac{\varepsilon}{\sqrt{n}}  \}\le \Big( \frac{C\varepsilon}{k}\Big)^{\gamma k^{2}}+e^{-cn},\nonumber
	\end{align}
	where $C, c$ are positive constants depending only on $a, b, K, \alpha, \gamma$.
	
\end{mytheo}

\begin{myrem}\label{Rem_Mainresult1_remark}
	(i) Without assuming $\mathcal{L}(\xi, a)<b$,  we cannot have a big $s_{n-k+1}(M)$. In particular, if $\xi$ does not satisfy this condition, all the entries of $M$ will take one value with very large probability, in which case $s_{n-k+1}(M)$ would be $0$.
	
	(ii) The condition \eqref{condition_in_Theo1} is very general. There are many cases satisfying this condition. For example, if  $\textsf{E}\xi^{2}<\infty$, then $M$ satisfies \eqref{condition_in_Theo1} (see Lemma \ref{Lem_moment_to_levy_concentration} below for details). We also refer readers to Example 1.4 in \cite{Livshys_JAM} for an interesting case where $M$ satisfies \eqref{condition_in_Theo1} but $\textsf{E}\xi^{2}=\infty$.
\end{myrem}

\begin{proof}[Proof of Theorem \ref{Theo_Main1}]
	
	Without loss of generality, we only prove Theorem \ref{Theo_Main1} when $n$ is large enough. In particular, we  assume Theorem \ref{Theo_Main1} is true when $n\ge n_{0}$. Note that 
	\begin{align}
		\textsf{P}\{s_{n-k+1}(M)\le \frac{\varepsilon}{\sqrt{n}}   \}\le \textsf{P}\{s_{n}(M)\le \frac{\varepsilon}{\sqrt{n}}   \}\le C\varepsilon+e^{-cn},\nonumber
	\end{align}
	where the last inequality is due to Lemma \ref{Lem_smallest_singular_heavy_tailed_entries}.
	Hence, when $n<n_{0}$, we have for $\varepsilon\le e^{-cn_{0}}/C$
	\begin{align}
		\textsf{P}\{s_{n-k+1}(M)\le \frac{\varepsilon}{\sqrt{n}}   \}\le 2e^{-cn}.\nonumber
	\end{align}
	As for the case $n<n_{0}$ and $\varepsilon> e^{-cn_{0}}/C$, the desired bound is trivial.
	
	
Following the same line as in the proof of Theorem \ref{Theo_Main1}, we have	for $\varepsilon>0$ and $1\le l\le k-1$
\begin{align}
	\textsf{P}\{s_{n-k+1}(M)\le \frac{\varepsilon}{\sqrt{n}}  \}\le \binom{n}{k}\binom{n}{l}\textsf{P}\Big\{ \sum_{m=1}^{l}\text{dist}(\boldsymbol{c}_{i_{m}}, H)^{2}\le \Big(\frac{C_{1}k\varepsilon}{k-l}\Big)^{2}  \Big\}.\nonumber
\end{align}
	Here $\boldsymbol{c}_{i_{m}}=\textnormal{Col}_{i_{m}}(M)$ and $H=\textnormal{span}\big\{\boldsymbol{c}_{i_{l+1}}, \cdots, \boldsymbol{c}_{i_{n}}\big\}$.
	Consider the event
	\begin{align}
		\Omega:=\{ D_{c_{1}\sqrt{n}, c_{2}}^{(1)}(H^\perp)>C\sqrt{n}e^{\frac{c_{3}n}{l}}  \}.\nonumber
	\end{align}
	Let us condition on a realization of $H$ in $\Omega$. Lemma \ref{Lem_distance_to_general_subspace_random_rounding}  yields for $\varepsilon>\exp(-\frac{c_{3}n}{l})/C$
	\begin{align}
		\textsf{P}\{\text{dist}(\boldsymbol{c}_{i_1}, H(\omega))\le \varepsilon\sqrt{l}   \}\le (C_{2}\varepsilon)^{l}+C_{3}^{l}e^{-c_{4}n}.\nonumber
	\end{align}
	As for the case $\varepsilon\le \exp(-\frac{c_{3}n}{l})/C$, we have 
	\begin{align}
		\textsf{P}\{\text{dist}(\boldsymbol{c}_{i_1}, H(\omega))\le \varepsilon\sqrt{l}\}  &\le 
		\textsf{P}\{\text{dist}(\boldsymbol{c}_{i_1}, H(\omega))\le\frac{\sqrt{l}}{C}e^{-\frac{c_{3}n}{l}} \}\nonumber\\
		&  \le (\frac{C_{2}}{C})^{l}e^{-c_{3}n}+C_{3}^{l}e^{-c_{4}n}.\nonumber
	\end{align}
	Note that $l\le k-1<n^{\alpha}, 0<\alpha<1$. Hence, when $n$ is sufficiently large enough, we have for $\varepsilon>0$
	\begin{align}\label{Eq_Proof1_distance}
		\textsf{P}\{\text{dist}(\boldsymbol{c}_{i_1}, H(\omega))\le \varepsilon\sqrt{l}   \}\le (C_{2}\varepsilon)^{l}+e^{-c_{5}n}.
	\end{align} 
	
	Due to the independence of random variables $\text{dist}(\boldsymbol{c}_{i_{1}}, H(\omega)), \cdots, \text{dist}(\boldsymbol{c}_{i_{l}}, H(\omega))$, a tensorization (as in the proof of Theorem \ref{Theo_Main2}) yields for $t> 0$
	\begin{align}
		\textsf{P}\Big\{ \sum_{m=1}^{l}\text{dist}(\boldsymbol{c}_{i_{m}}, H(\omega))^{2}\le l^{2}t^2  \Big\}\le (C_{3}t)^{l^2}+e^{-c_{6}nl}.\nonumber
	\end{align}
	
	Lemma \ref{Lem_prob_LCD_1} yields that $\textsf{P}\{\Omega^{c}  \}\le e^{-c_{7}n}$. Hence, we have for $\varepsilon>0$
	\begin{align}
		&\textsf{P}\Big\{ \sum_{m=1}^{l}\text{dist}(\boldsymbol{c}_{i_{m}}, H)^{2}\le \Big(\frac{C_{1}k\varepsilon}{k-l}\Big)^{2}  \Big\}\nonumber\\
		\le& \textsf{P}\Big\{ \sum_{m=1}^{l}\text{dist}(\boldsymbol{c}_{i_{m}}, H)^{2}\le \Big(\frac{C_{1}k\varepsilon}{k-l}\Big)^{2}, \Omega  \Big\}+\textsf{P}\{\Omega^{c} \}\nonumber\\
		\le & \Big(\frac{C_{4}k\varepsilon}{l(k-l)}  \Big)^{l^{2}}+e^{-c_{8}n}.\nonumber
	\end{align}
	
	Hence, taking the union bound we have for $\varepsilon>0$
	\begin{align}\label{Eq_proof_main_trans}
		\textsf{P}\{s_{n-k+1}(M)\le \frac{\varepsilon}{\sqrt{n}}  \}\le& \binom{n}{k}\binom{n}{l}\textsf{P}\Big\{ \sum_{m=1}^{l}\text{dist}(\boldsymbol{c}_{i_{m}}, H)^{2}\le \Big(\frac{C_{1}k\varepsilon}{k-l}\Big)^{2}  \Big\}\nonumber\\
		\le &n^{k+l}\Big( \big(\frac{C_{4}k\varepsilon}{l(k-l)}  \big)^{l^{2}}+e^{-c_{8}n}  \Big).
	\end{align}
	Letting $l=\sqrt{\gamma}k, 0<\gamma<1$ ($l\le k-1$ when $n$ is large enough), we have by adjusting the absolute constant
	\begin{align}
		\textsf{P}\{s_{n-k+1}(M)\le \frac{\varepsilon}{\sqrt{n}}  \}\le \Big( \frac{C_{5}(\gamma)\varepsilon}{k}\Big)^{\gamma k^{2}}+e^{-c_{9}n},\nonumber
	\end{align}
	which concludes the proof.
	
\end{proof}

\subsection{Bounded density case}

\begin{mytheo}\label{Theo_main3}
	Let $M=(\xi_{ij})_{n\times n}$ be a random matrix with i.i.d. entries and  assume further that the density functions of $\xi_{ij}$ are bounded by $K_{1}$. Then for any fixed $\gamma\in(0, 1)$, when $n$ is large enough, we have for $\varepsilon>0$ and $\log n\le  k\le  n$
	\begin{align}
		\textsf{P}\{s_{n-k+1}(M)\le \frac{\varepsilon}{\sqrt{n}}  \}\le \big(\frac{CK_{1}\varepsilon}{k}  \big)^{\gamma k^{2}},\nonumber
	\end{align}
	where $C$ is a positive constant depending only on the parameters $K_{1}$.
\end{mytheo}

To obtain the local Marchenko-Pastur law at the hard edge of sample covariance matrices, Cacciapuoti, Maltsrv, and Schlein \cite{Cacciapuoti_JNP} also estimated the number of the singular values close to zero. In particular, let $I$ be the interval $[0, k\varepsilon/\sqrt{n}]$. Denote by $N_{I}$ the number of singular values of $M$ in $I$. In the same setting of Theorem \ref{Theo_main3}, they showed that 
\begin{align}\label{Eq_intro_}
	\textsf{P}\{N_{I}\ge k  \}=O(\varepsilon^{ck}).
\end{align}
Compared with this result, Theorem \ref{Theo_main3} yields that the propability of $\{N_{I}\ge k  \}$ is $O(\varepsilon^{ck^{2}})$, which is much better than the estimate in \eqref{Eq_intro_}.

\begin{proof}[Proof of Theorem \ref{Theo_main3}]
	By virtue of \eqref{Eq_proof_main_trans}, we have
	\begin{align}\label{Eq_proof_main3}
		\textsf{P}\big\{  s_{n-k+1}(M)\le \frac{\varepsilon}{\sqrt{n}}    \big\}\le n^{k+l}\textsf{P}\Big\{\sum_{m=1}^{l}\textnormal{dist}(\boldsymbol{c}_{i_{m}}, H)^{2}\le \big( \frac{C_{1}k\varepsilon}{k-l} \big)^{2}  \Big\},
	\end{align}
	where $\boldsymbol{c}_{i}$ is the $i$-th column of $M$, $H=\textnormal{span}\{\boldsymbol{c}_{i_{l+1}}, \cdots, \boldsymbol{c}_{i_{n}}\}$ and $1\le l\le k-1$.

	Condition on a realization of $H$, then Lemma \ref{Lem_smallballprobability_bounded_density} yields that
	\begin{align}
		\textsf{P}\{\textnormal{dist}(\boldsymbol{c}_{i_{1}}, H(\omega))\le \varepsilon\sqrt{l}  \}&= \int_{\{\textnormal{dist}(\boldsymbol{c}_{i_{1}}, H(\omega))\le \varepsilon\sqrt{l}  \}}1\,d\textsf{P}\nonumber\\
		&\le (C_{2}K_{1})^{l}\cdot \frac{\pi^{l/2}}{\Gamma(l/2+1)}(\varepsilon\sqrt{l})^{l}\le (C_{3}K_{1}\varepsilon)^{l}.\nonumber
	\end{align}
	Note that $\{\textnormal{dist}(\boldsymbol{c}_{i_{1}}, H(\omega)), \cdots, \textnormal{dist}(\boldsymbol{c}_{i_{l}}, H(\omega))\}$ is a sequence of independent random variables. Hence, following the tensorization argument in the proof of Theorem \ref{Theo_Main2}, we have for $t\ge 0$
	\begin{align}
		\textsf{P}\Big\{\sum_{m=1}^{l}\textnormal{dist}(\boldsymbol{c}_{i_{m}}, H(\omega))^{2}\le l^{2}t^{2}   \Big\}\le (C_{4}K_{1}t)^{l^{2}}.\nonumber
	\end{align}
	
	Note that $\boldsymbol{c}_{i_{1}}, \cdots, \boldsymbol{c}_{i_{l}}$ are independent of $H$. Hence, we have for $t\ge 0$
	\begin{align}
		\textsf{P}\Big\{\sum_{m=1}^{l}\textnormal{dist}(\boldsymbol{c}_{i_{m}}, H)^{2}\le l^{2}t^{2}   \Big\}\le (C_{4}K_{1}t)^{l^{2}}.\nonumber
	\end{align}
	Hence, \eqref{Eq_proof_main3} yields that
	\begin{align}
		\textsf{P}\big\{s_{n-k+1}(M)\le \frac{\varepsilon}{\sqrt{n}}  \big\}\le n^{k+l}\big( \frac{C_{5}K_{1}k\varepsilon}{l(k-l)}  \big)^{l^{2}}.\nonumber
	\end{align}
	Taking $l=\sqrt{\gamma}k, 0<\gamma<1 $, we have when $n$ is large enough
	\begin{align}
		\textsf{P}\big\{ s_{n-k+1}(M)\le \frac{\varepsilon}{\sqrt{n}}   \big\}\le \big(\frac{C_{6}(\gamma)K_{1}\varepsilon}{k}  \big)^{\gamma k^{2}},\nonumber
	\end{align}
	which concludes the proof.
	
\end{proof}

At the end of this section, we shall illustrate the relationship between the three types of random matrices discussed in this paper. 

Let $M=(\xi_{ij})$ be an $n\times n$ random matrix and assume that $\xi_{ij}\stackrel{i.i.d.}{\sim}\xi$. 
Define the following set of random matrices:
\begin{align}
	&A:=\{M: M \,\,\text{satisfies the condition \eqref{condition_in_Theo1}}  \},\nonumber\\
	&B:=\{M: \xi \,\,\text{is a centered non-constant subgaussian random variable} \},\nonumber\\
	&C:=\{M: \xi \,\,\text{has the bounded density function}  \}.\nonumber
\end{align}
These sets have the following relationship:
\begin{figure}[h]
	\centering
	\includegraphics[width=6cm]{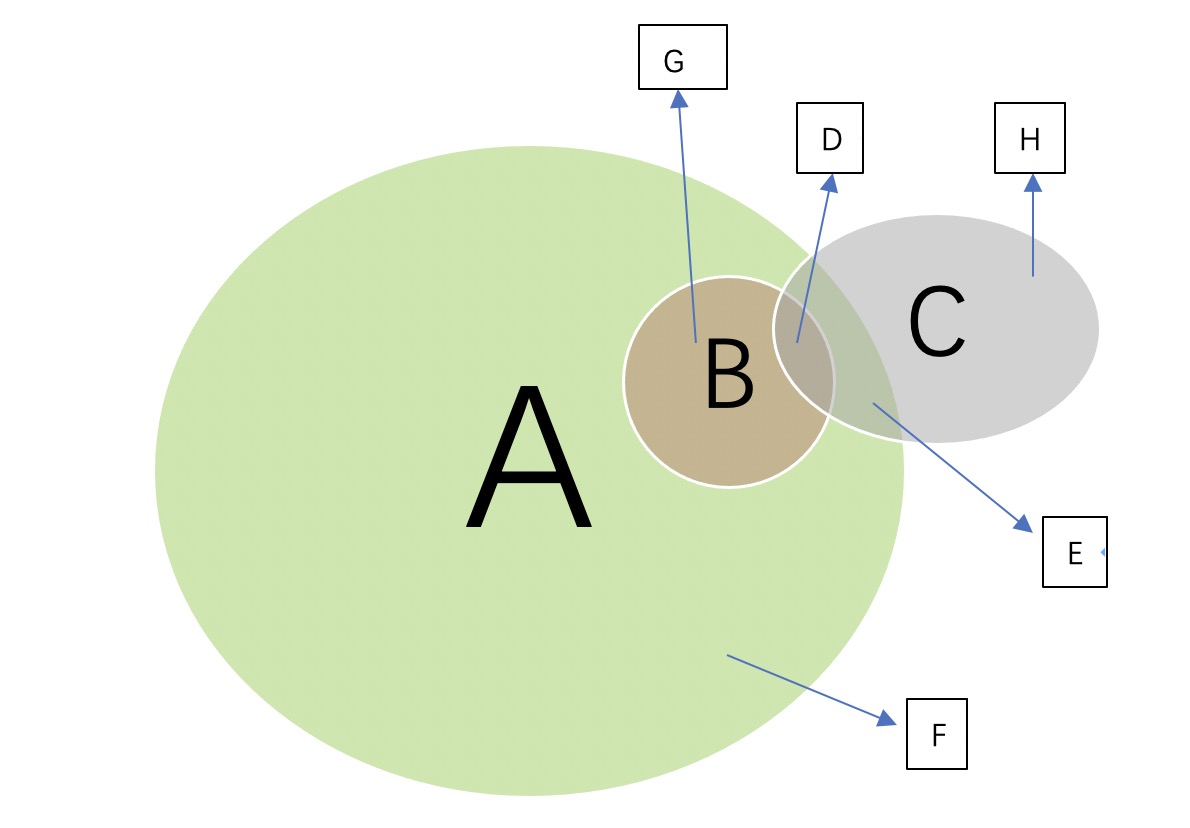}
	\label{fig:example}
\end{figure}

For the sake of reading, we give specific examples  in each set.
\begin{itemize}
	\item $M\in D:=A\cap B\cap C$, if $\xi$ is a centered gaussian random variable;
	\item $M\in E:=(A\cap C)-B$, if $\xi$ is an exponential random variable;
	\item $M\in F:=A-B-C$, if $\xi$ is a geometric random variable;
	\item $M\in G:=B-C$, if $\xi$ is a Rademacher random variable;
	\item $M\in H:=C-A$, if $\xi$ is a Cauchy random variable.
\end{itemize}

\section{Appendix}
In this section, we mainly show that if a random variable $\xi$ satisfies $\textsf{E}\xi^{2}<\infty$, then $\xi$ has a bounded L\'{e}vy concentration function.
\begin{mylem}\label{Lem_moment_to_levy_concentration}
	Let $\xi$ be a random variable such that $\textsf{E}\xi^{2}\le K_{2}$. There exist constants $a>0, b\in (0, 1)$ such that
	\begin{align}
		\mathcal{L}(\xi, a)<b.\nonumber
	\end{align}
\end{mylem}
\begin{proof}
	Without loss of generality, assume that $\textsf{E}\xi=0$ and $\textsf{E}\vert\xi\vert^{3/2}=1$. Given a constant $m>0$, we make the following decomposition:
	\begin{align}
		\xi=\xi\cdot\mathbb{I}_{\{\vert\xi\vert\le m\}}+\xi\cdot\mathbb{I}_{\{\vert\xi\vert> m\}}=:\xi_{1}+\xi_{2},\nonumber
	\end{align}
where $\mathbb{I}_{A}$ is an indicator function of the set $A$.
Then we have for any $z\in \mathbb{R}$
	\begin{align}\label{Eq_appendix_decom}
		\textsf{P}\{\vert\xi-z\vert\le 1/16  \}\le \textsf{P}\{\vert\xi_{1}-z\vert\le 1/16 \}+\textsf{P}\{ \vert\xi\vert> m  \}.
	\end{align}
For the second addition on the right side of \eqref{Eq_appendix_decom}, we have by Markov's inequality
\begin{align}
    \textsf{P}\{ \vert\xi\vert> m  \}\le \frac{\textsf{E}\vert\xi\vert^{2}}{m^{2}}\le \frac{K_{2}}{m^{2}},\nonumber
\end{align}

Next we turn to bound the other addition term of \eqref{Eq_appendix_decom}. Denote
\begin{align}
	p=\textsf{P}\{\vert\xi_{1}-z\vert\le 1/16  \}.\nonumber
\end{align}
For the case $\vert z\vert\le 3/16$, we have 
\begin{align}\label{Eq_appendix_case_z_1}
	\textsf{E}\vert\xi_{1}\vert^{3/2}&\le \textsf{P}\{\vert\xi_{1}-z\vert\le 1/16  \}\cdot(\frac{1}{4})^{3/2}+\textsf{P}\{\vert\xi_{1}-z\vert>1/16  \}m^{3/2}\nonumber\\
	&=\frac{p}{8}+m^{3/2}(1-p).
\end{align}
Note that
\begin{align}
	\textsf{E}\vert\xi\vert^{3/2}&=\textsf{E}\big(\vert\xi\vert^{3/2}(\mathbb{I}_{\vert\xi\vert\le m}+\mathbb{I}_{\vert\xi\vert> m})\big)=\textsf{E}\vert\xi_{1}\vert^{3/2}+\textsf{E}\vert\xi_{2}\vert^{3/2}\nonumber\\
	&=\textsf{E}\vert\xi_{1}\vert^{3/2}+\int_{0}^{m^{3/2}}\textsf{P}\{\vert\xi_{2}\vert^{3/2} >t \}\, dt+\int_{m^{3/2}}^{\infty}\textsf{P}\{\vert\xi_{2}\vert^{3/2} >t \}\, dt\nonumber\\
	&\le \textsf{E}\vert\xi_{1}\vert^{3/2}+m^{3/2}\textsf{P}\{\vert\xi_{2}\vert^{3/2}>m^{3/2} \}+K_{2}\int_{m^{3/2}}^{\infty}t^{-4/3}\, dt\nonumber\\
	&\le \textsf{E}\vert\xi_{1}\vert^{3/2}+\frac{K_{2}}{\sqrt{m}}+\frac{3K_{2}}{\sqrt{m}},\nonumber
\end{align}
where the first inequality is due to the fact $\{\vert\xi_{2}\vert>0 \}=\{\vert\xi_{2}\vert>m \}$ and  Markov's inequality.
Hence, $\textsf{E}\vert\xi_{1}\vert^{3/2}\ge 1-4K_{2}/\sqrt{m}$.  By virtue of \eqref{Eq_appendix_case_z_1}, we have for $\vert z\vert\le 3/16$
\begin{align}
	p\le \frac{m^{2}+4K_{2}-\sqrt{m}}{m^{2}-\sqrt{m}/8}=:f_{1}(m).\nonumber
\end{align}
For the case $\vert z\vert>3/16$, we have 
\begin{align}\label{Eq_appendix_case2}
	\vert\textsf{E}\xi_{1}\vert\ge \textsf{E}\xi_{1}&\ge \textsf{P}\{\vert\xi_{1}-z\vert\le 1/16  \}(z-1/16)+\textsf{P}\{\vert\xi_{1}-z\vert>1/16  \}(-m)\nonumber\\
	&\ge p/8-(1-p)m.
\end{align}
Note that 
\begin{align}
	\vert \textsf{E}\xi_{1}\vert\le &=\vert \textsf{E}\xi_{2}\vert\le \int_{0}^{\infty}\textsf{P}\{\vert\xi_{2}\vert>t  \}\, dt\nonumber\\
	&\le m\textsf{P}\{ \vert\xi\vert>m  \}+\int_{m}^{\infty}\textsf{P}\{\vert\xi\vert>t  \}\, dt\le \frac{2K_{2}}{m}.\nonumber
\end{align}
Hence, by virtue of \eqref{Eq_appendix_case2}, we have for $\vert z\vert>3/16$
\begin{align}
	p\le \frac{m^{2}+2K_{2}}{m^{2}+m/8}=:f_{2}(m).\nonumber
\end{align}
Then, for any $z\in\mathbb{R}$, we have by virtue of \eqref{Eq_appendix_decom}
\begin{align}
	\mathcal{L}(\xi, 1/16)\le \max\big(f_{1}(m), f_{2}(m)\big)+\frac{K_{2}}{m^{2}}.\nonumber
\end{align}
Note that
\begin{align}
	f_{1}(m)=1-\frac{7\sqrt{m}/8-4K_{2}}{m^{2}-\sqrt{m}/8}\le 1-\frac{7\sqrt{m}/8-4K_{2}}{m^{2}}.\nonumber
\end{align}
and 
\begin{align}
	f_{2}(m)=1-\frac{m/8-2K_{2}}{m^{2}+m/8}\le 1-\frac{m/8-2K_{2}}{m^{2}}.\nonumber
\end{align}
When $m>49$, we have $f_{1}(m)>f_{2}(m)$. Hence, for $m>49$
\begin{align}
	\mathcal{L}(\xi, 16)\le f_{1}(m)+\frac{K_{2}}{m^{2}}\le 1-\frac{7\sqrt{m}/8-5K_{2}}{m^{2}}.\nonumber
\end{align}
We conclude the proof by letting
\begin{align}
	m=\max\big(49, (\frac{40K_{2}}{7})^{2}+1   \big).\nonumber
\end{align}

\end{proof}

\textbf{Acknowledgment:} Su was partly supported by the National Natural Science Foundation of China  (No. 12271475 and U23A2064). Wang was partly supported by Shandong Provincial Natural Science Foundation (No. ZR2024MA082) and the National Natural Science Foundation of China (No. 12071257).


\end{document}